\documentclass[12pt]{amsart}
\usepackage[pagewise]{lineno}
\usepackage[left=2.8cm,right=2.8cm,top=2.9cm,bottom=2.9cm]{geometry}

\usepackage{amsfonts}
\usepackage{color}
\usepackage{graphicx}
\usepackage[a4paper,bookmarksnumbered,colorlinks, linkcolor=blue, citecolor=red, pagebackref, bookmarks, breaklinks]{hyperref}

\newtheorem{thm}{Theorem}[section]
\newtheorem{cor}[thm]{Corollary}
\newtheorem{lema}[thm]{Lemma}
\newtheorem{prop}[thm]{Proposition}
\theoremstyle{definition}

\theoremstyle{remark}

\numberwithin{equation}{section}

\title[Eigenvalue homogenization for the $p-$Laplacian with $L^q$ weights]{Eigenvalue homogenization for the $p-$Laplacian with rapidly oscillating $L^q$ weights}

\author{Ariel Salort}
\address{Ariel Salort \newline Departamento de Matematicas y Ciencia de Datos, Universidad San Pablo-CEU, CEU Universities, Urbanizacion Montepr\'incipe, 28660 Boadilla del Monte, Madrid, Spain.}
\email{\tt ariel.salort@ceu.es}
\urladdr{https://sites.google.com/view/amsalort/}

\begin{document}



\begin{abstract}
In this article, we study the convergence rates of variational eigenvalues of the $p$-Laplacian with rapidly oscillating weights and potentials belonging in a suitable $L^q$ space. Our analysis covers both Dirichlet and Neumann boundary conditions, and we derive explicit estimates in terms of the eigenvalue index $k \in \mathbb{N}$ and the oscillation parameter $\varepsilon > 0$. These results are obtained through a detailed examination of certain oscillatory integrals and extend previously known results in the literature.
\end{abstract}

\maketitle

\section{Introduction}
The homogenization of elliptic operators with rapidly oscillating coefficients has been and continues to be a highly active research area, given its wide range of applications in fields such as composite materials, optimal design, and shape optimization. To gain an introduction and comprehensive understanding of this theory, we recommend referring to \cite{A,CZ,JKO,k1,MV,OSY,SV}. A precise comprehension of the behavior of eigenpairs in response to small perturbations is particularly essential for modeling periodic phenomena mathematically. For a detailed description of various models and their treatment, we recommend referring to \cite{A,OSY} and their respective references.

The purpose of this article is to conduct a comprehensive analysis of the behavior of eigenvalues in a nonlinear equation that incorporates rapidly oscillating potentials and weight functions with appropriate integrability.  Specifically, for each value of $\varepsilon > 0$, we consider the Dirichlet problem ($\mathcal{D}_\varepsilon$) given by:
\begin{align}\tag{$\mathcal{D}_\varepsilon$}   \label{eq.p.eps} 
\begin{cases}
-\Delta_p u_\varepsilon + V_\varepsilon |u|_\varepsilon^{p-2} u_\varepsilon= \lambda_\varepsilon \rho_\varepsilon |u_\varepsilon|^{p-2}u_\varepsilon &\text{ in } \Omega\\
u_\varepsilon=0 & \text{ on } \partial\Omega.
\end{cases}
\end{align}
where $\Omega$ is a bounded open set in $\mathbb R^n$,  $n\geq 1$ and with Lipschitz boundary.  Here, for $p>1$, $\Delta_p u :=\mathop{\text{\normalfont div}}(|\nabla u|^{p-2} u)$ denotes the $p-$Laplacian of $u$, and $\lambda_\varepsilon\in \mathbb R$ is an eigenvalue of \eqref{eq.p.eps} with eigenfunction $u_\varepsilon$. The sequences of weight functions $\{\rho_\varepsilon\}_{\varepsilon>0}$ and potential functions $\{V_\varepsilon\}_{\varepsilon>0}$ are uniformly positive  and $\rho_\varepsilon,V_\varepsilon \in L^q(\Omega)$ for  some appropriate number $q$ (see conditions \eqref{rel1}--\eqref{rel2}).

It is well-known that the sequence of uniformly bounded functions in $L^q(\Omega)$ given by $\{\rho_\varepsilon\}_{\varepsilon>0}$ converges when $\varepsilon \to 0$: there exists a function $\rho_0\in L^q(\Omega)$ such that
\begin{align*}
\rho_\varepsilon \rightharpoonup \rho_0 &\text{ weakly in  } L^q(\Omega), \text{when }q<+\infty,\\
\rho_\varepsilon \stackrel{*}{\rightharpoonup} \rho_0 &\text{ weakly* in  } L^\infty(\mathbb R^n), \text{when }q=+\infty.
\end{align*}
Analogously, there exists $V_0\in L^q(\Omega)$ for which the same convergences hold for the sequence of uniformly bounded function $\{V_\varepsilon\}_{\varepsilon>0}\in L^q(\Omega)$. Therefore, the limit problem of $\eqref{eq.p.eps}$ when $\varepsilon\to 0$ is  given by
\begin{align}\tag{$\mathcal{D}_0$}   \label{eq.p.0} 
\begin{cases}
-\Delta_p u_0 +  V_0 |u_0|^{p-2}u_0 = \lambda_0 \rho_0 |u_0|^{p-2}u_0 &\text{ in } \Omega\\
u_0=0 & \text{ on } \partial\Omega.
\end{cases}
\end{align}

In the case where the sequence $\{\rho_\varepsilon\}_{\varepsilon>0}$ is defined in terms of a $Q$-periodic function $\rho\in L^q(Q)$, where $Q$ is the unit cube in $\mathbb R^n$, as $\rho_\varepsilon(x):=\rho\left(\frac{x}{\varepsilon}\right)$, the limit function $\rho_0$ can be characterized as the average of $\rho$ over $Q$, i.e.,
$$
\rho_0=\overline\rho :=\frac{1}{|Q|}\int_Q \rho(x)\,dx.
$$
The same assertion holds when $\{V_\varepsilon\}_{\varepsilon>0}$ is given in terms of a periodic function $V\in L^q(Q)$. For more details on this topic, we refer the reader to Theorem 2.6 in \cite{CD}.

The existence of a sequence $\{\lambda_{k,\varepsilon}\}_{k\in\mathbb N}$ of variational eigenvalues for \eqref{eq.p.eps} is deduced using minimax methods for any given $\varepsilon>0$, see section \ref{sec.prel}. These eigenvalues satisfy $0<\lambda_{1,\varepsilon}<\lambda_{2,\varepsilon}\leq \lambda_{3,\varepsilon}\leq \cdots \nearrow +\infty$. Similarly, there exists a sequence $\{\lambda_{k,0}\}_{k\in\mathbb N}$ of variational eigenvalues for \eqref{eq.p.0} such that $0<\lambda_{1,0}<\lambda_{2,0}\leq \lambda_{3,0}\leq \cdots \nearrow +\infty$.

A natural question that arises is whether the eigenvalues of \eqref{eq.p.eps} converge to the eigenvalues of \eqref{eq.p.0} as $\varepsilon\to 0$. Such problems are closely linked to the field of \emph{homogenization theory} and the \emph{complete continuity of eigenpairs}. (For further information, see \cite{A, LY, OSY} and related literature.) Once the convergence $\lambda_{k,\varepsilon} \to \lambda_{k,0}$ has been established, a significant issue is to estimate the uniform rate of convergence in terms of $\varepsilon$ and $k$.

Under the same assumptions, when $\Omega\subset \mathbb R^n$, $n\geq 2$, we also   consider the eigenvalue problem with Neumann boundary condition
\begin{align}\tag{$\mathcal{N}_\varepsilon$}   \label{eq.n.eps} 
\begin{cases}
-\Delta_p u_\varepsilon + V_\varepsilon |u|_\varepsilon^{p-2} u_\varepsilon= \lambda_\varepsilon \rho_\varepsilon |u_\varepsilon|^{p-2}u_\varepsilon &\text{ in } \Omega\\
|\nabla u_\varepsilon|^{p-2}\nabla u_\varepsilon \cdot \eta =0 & \text{ on } \partial\Omega,
\end{cases}
\end{align}
where $\eta$ denotes the outward unit normal. When $\varepsilon\to 0$,  the following limit problem is obtained
\begin{align}\tag{$\mathcal{N}_0$}   \label{eq.n.0} 
\begin{cases}
-\Delta_p u_0 +  V_0 |u_0|^{p-2}u_0 = \lambda_0 \rho_0 |u_0|^{p-2}u_0 &\text{ in } \Omega\\
|\nabla u_0|^{p-2}\nabla u_0 \cdot \eta =0 & \text{ on } \partial\Omega.
\end{cases}
\end{align}

Problems of this kind have many practical applications in fields such as population biology, the Schr\"odinger operator, and shape optimization (see \cite{A, LY, OSY}). The literature provides a wealth of information regarding various properties of the eigenvalues and eigenfunctions of these equations. Notably, \cite{BH, CR} offer a broad range of weights.

Let us now review some of the existing results regarding the convergence of eigenvalues.  In the linear case (i.e., $p=2$) the convergence $\lambda_{k,\varepsilon}\to \lambda_{k,0}$ as $\varepsilon\to 0$ was obtained in different settings of $V_\varepsilon,\rho_\varepsilon\in L^q(\Omega)$: when $q=\infty$ and $V_\varepsilon=0$ with Dirichlet boundary condition, convergence of the spectrum is obtained for instance in \cite{A}; the case of Neumann boundary condition was studied in \cite{LY}; for the Steklov case we refer to \cite{S}. In the Dirichlet settings, the case $V_\varepsilon=0$ and $q>\frac{n}{2}$ was dealt in \cite{WYZ}; the case of integrable functions $V_\varepsilon$ and $\rho_\varepsilon$ in the one-dimensional equation was treated in \cite{CQ}.

In the linear periodic case with $q=\infty$ and $V_\varepsilon=0$, an estimate of the order of convergence can be found in \cite{OSY} both for Dirichlet and Neumann boundary conditions, where the authors prove that $|\lambda_{k,\varepsilon} - \lambda_{k,0}| \leq C\sqrt{\varepsilon} \lambda_{k,\varepsilon} \lambda_{k,0}^2 /(1-\lambda_{k,0}\beta_{k,\varepsilon})$ where $0\leq \beta_{k,\varepsilon}<\lambda_{k,0}^{-1}$ and  $\lim_{\varepsilon\to 0}\beta_{k,\varepsilon}=0$  with $C$ a positive constant independent of $\varepsilon$ and $k$. For the same problem in the Dirichlet case, in \cite{SV}, the authors prove that rate is of order $\varepsilon$: $|\lambda_{k,\varepsilon} - \lambda_{k,0}| \leq C_k \varepsilon$ where $C_k$ is a positive constant depending of $k$ but independent of $\varepsilon$. See also \cite{KLS} for some results in this line of research. Finally, under the same assumptions, when  $n=1$ in \cite{CZ} the authors prove that
$|\lambda_{k,\varepsilon} - \lambda_{k,0}| \leq C k^3 \varepsilon$ where $C>0$ is independent of $k$ and $\varepsilon$. The linear problem with $\rho_\varepsilon=0$ and weight $V_\varepsilon$ in $L^q(\Omega)$, $q\geq  3n/(n+2)$ was treated in \cite{WZ} where the authors obtain that $|\lambda_{k,\varepsilon}-\lambda_{k,0}|\leq C\varepsilon k^{\frac{n}{2}(1+\gamma)}$ for some $\gamma>0$ and $C$ independent of $k$ and $\varepsilon$. Finally, when $V_\varepsilon=0$ and $q=\infty$, rates of the convergence for the eigenvalues with  Steklov case were obtained in \cite{S}.

In the nonlinear case (i.e. $p\neq 2$),  when $q=1$ and $p>n$, in \cite{CR} convergence of   eigenvalues was proved in the one-dimensional case; when $q>\frac{n}{p}$ and $1<p\leq n$, in \cite{CR} the authors also obtain   convergence of the first two eigenvalues; in the periodic case, when $V_\varepsilon=0$ and $q=\infty$, rates of the convergence were obtained in \cite{FPS0,FPS2, S2} for the Dirichlet and Neumann case, giving  that $|\lambda_{k,\varepsilon}-\lambda_{k,0}|\leq C\varepsilon k^\frac{p+1}{n}$ with $C>0$ independent of $k$ and $\varepsilon$; finally,  the case of bounded sign changing weights was treated in \cite{FPS, FPS3}.

The primary objective of this manuscript is to investigate the convergence rates of eigenvalues in the homogenization problem for a broad range of values of $q$, where $V_\varepsilon,\rho_\varepsilon\in L^q (\Omega)$. To achieve this, we conduct a rigorous analysis of rapidly oscillating integrals and their relationship with the variational formulation of eigenvalues.

We describe now with more precision our main results. To this end, for each $k\in \mathbb N$ and $\varepsilon>0$ denote by $\lambda_{k,\varepsilon}$ the $k-$th variational eigenvalues at level $\varepsilon$  and by $\lambda_{k,0}$, the $k-$th variational eigenvalue of limit problems, that is, eigenvalues of \eqref{eq.p.eps},\eqref{eq.n.eps} or \eqref{eq.p.0},\eqref{eq.n.0} as appropriate.

Our first result concerns general sequences of weight functions. Let $\{\rho_\varepsilon\}_{\varepsilon>0}$ and $\{V_\varepsilon\}_{\varepsilon>0}$ be general sequences in $L^q(\Omega)$ (i.e., without periodicity assumptions)  with 
\begin{align} 
\frac{n}{p} < q\leq \infty &\quad \text{ when }1<p\leq n  \label{rel1}\\
q=1 &\quad \text{ when } p>n \label{rel2}
\end{align}
converging weakly in $L^q(\Omega)$ when $q<\infty$ (resp., weakly* in $L^\infty(\Omega)$ when $q=\infty$) to $\rho_0$ and $V_0$.
In Theorem \ref{teo.d.sinorden} (for the Dirichlet case) and Theorem \ref{teo.n} (for the Neumann case) we prove that for each $k\in \mathbb N$, the $k-$th variational eigenvalue of \eqref{eq.p.eps} and \eqref{eq.n.eps} converge as $\varepsilon\to 0$ in the sense that
$$
\lim_{\varepsilon\to 0} \lambda_{k,\varepsilon} = \lambda_{k,0}.
$$
When $\{\rho_\varepsilon\}_{\varepsilon>0}$ and $\{V_\varepsilon\}_{\varepsilon>0}$ have the form $\rho_\varepsilon(x):=\rho(\tfrac{x}{\varepsilon})$ and $V_\varepsilon(x):=V(\tfrac{x}{\varepsilon})$, where $\rho$ and $V$ are positive $Q-$periodic functions in the unit cube $Q$ of $\mathbb R^n$, these sequences become rapidly oscillating as $\varepsilon\to 0$.  In this case, a precise estimate of  related rapidly oscillating integrals (see Propositions \ref{prop.dirichlet} and \ref{prop.neumann}) enable us to obtain the  convergence rates of eigenvalues. Indeed, in Theorem \ref{teo.d} we prove the following result in the Dirichlet case:
\begin{align} \label{eqqi}
|\lambda_{k,0}-\lambda_{k,\varepsilon}| &\leq  
C  \varepsilon^\alpha \left(  \lambda_{k,0} \|V_\varepsilon - \bar V\|_{L^q(\Omega)} +     \lambda_{k,0}^2 \|\rho_\varepsilon - \bar \rho\|_{L^q(\Omega)}   \right)
\end{align}
holds for all $0<\varepsilon<(2\textbf{C} \lambda_{k,0})^{-\frac{1}{\alpha}}$, where 
$$
\alpha=
\begin{cases}
p-\frac{n}{q} & \text{ when } \frac{n}{p}<q\leq \infty  \text{ and } 1<p\leq n\\
1-\frac{n}{p} & \text{ when } q=1 \text{ and } p>n,
\end{cases}
$$
and $C,\textbf{C}$ are computable positive constants depending only of $p$,  $\rho$, $V$ and $\Omega$, and independent of $k$ and $\varepsilon$. Moreover, by using precise estimates on the eigenvalues $\lambda_{k,\varepsilon}$ and $\lambda_{k,0}$, we obtain in Corollary \ref{cor.d} the rates of convergence in terms of $\varepsilon$ and $k$.

In the case of Neumann eigenvalues for $1<p\leq n$ and $\frac{n-1}{p-1}<q\leq \infty$, in  Theorem \ref{teo.n}  we prove that an estimate similar to \eqref{eqqi} holds with $\alpha=1-\frac{1}{q}$.

We stress that these results generalize many of the estimates on eigenvalue homogenization known in the literature.

We conclude the description of our result with a theorem regarding an eigenvalue problem in dimension one. In this case,  for positive $1-$periodic functions $a\in L^\infty$ and $\rho\in L^q$ with $p>1$ and $q\geq 1$, we consider, for each $\varepsilon>0$,  the following eigenvalue problem involving a  weighted $p-$Laplacian 
\begin{align} \label{eq.p.a.intro}
\begin{cases}
-(a(\tfrac{x}{\varepsilon})|u_\varepsilon|^{p-2}u_\varepsilon')'=\lambda_\varepsilon \rho(\tfrac{x}{\varepsilon}) |u_\varepsilon|^{p-2}u_\varepsilon &\quad \text{ in } (0,1)\\
u_\varepsilon(0)=u_\varepsilon(1)=0 
\end{cases}
\end{align}
for which, as $\varepsilon\to 0$ the following limit equation is obtained
\begin{align} \label{eq.p.a.lim.intro}
\begin{cases}
-(a^* | u_0|^{p-2}u_0 ')'=\lambda_0 \bar \rho |u_0|^{p-2}u_0 &\quad \text{ in } (0,1)\\
u_0(0)=u_0(1)=0 
\end{cases}
\end{align}
where $a^* = \left(\int_0^1 a(t)^{-\frac{1}{p-1}}\,dt \right)^{-(p-1)}$ and $\bar \rho=\int_0^1 \rho(t)\,dt$.

In this case, since this problem is related to a  weighted  oscillating operator, the   arguments applied in the previous results do not work. However, by means  of a suitable change of variables we can reduce problem \eqref{eq.p.a.intro} to an equation involving only the $p-$Laplacian and an oscillating weight function in the right hand side. Then, denoting by $\lambda_{k,\varepsilon}$ and $\lambda_{k,0}$ the $k-$th variational eigenvalue of \eqref{eq.p.a.intro} and \eqref{eq.p.a.lim.intro}, respectively, in Theorem \ref{teo.1d}  we prove that $\lambda_{k,\varepsilon} \to \lambda_{k,0}$ as $\varepsilon\to 0$ and moreover, 
\begin{align*}
|\lambda_{k,0}-\lambda_{k,\varepsilon}|\leq 
\begin{cases}
C L^{-p} \|a\|_{L^\infty(0,1)}^\frac{1}{p-1} \|\rho\|_{L^q(\Omega)}  k^{2p}\varepsilon  &\text{ if } \varepsilon^{-1}\in \mathbb N\\
C L^{-p} \left( \|a\|_{L^\infty(0,1)}^\frac{1}{p-1} \|\rho\|_{L^q(\Omega)}   k^{2p}\varepsilon +  k^p \varepsilon \right) &\text{ if } \varepsilon^{-1}\not \in \mathbb N
\end{cases}
\end{align*}
holds for  $\varepsilon<(2\textbf{C} \lambda_{k,0})^{-1}$, where $L=(a^*)^\frac{1}{1-p}$ and $C,{\bf C}$ are computable positive constants  independent of $k$ and $\varepsilon$.

Finally, we remark that all the results stated in this paper hold true if we replace $\Delta_p u_\varepsilon$, $\varepsilon>0$ by a quasi-linear operator of the form
$$
\mathcal{L}u_\varepsilon :=\mathop{\text{\normalfont div}}(a(x,\nabla u_\varepsilon))
$$
(i.e., independent of $\varepsilon$ in the first parameter) where, for any $\xi\in\mathbb R^n$ and $x\in\Omega$ the function $a\colon \Omega\times\mathbb R^n\to \mathbb R^n$ satisfies the following conditions: 
\begin{itemize}
\item[$h_0$)] $a(\cdot,\cdot)$ is a Carath\'eodory function odd in the second variable,
\item[$h_1$)] continuity: for some $c_1>0$ it holds that $c_1|\xi|^p\leq a(x,\xi)\cdot \xi$,
\item[$h_2$)] coercivity: for some $c_2>0$ it holds that $|a(x,\xi)|\leq c_2|\xi|^{p-1}$,
\item[$h_3$)] monotonicity: $0\leq (a(x,\xi_1)-a(x,\xi_2))\cdot (\xi_1-\xi_2)$,
\item[$h_4$)] $p-$homogeneity: $a(x,t\xi)=t^{p-1}a(x,\xi)$ for all $t>0$,
\item[$h_5$)] equi-continuity: 
$$|a(x,\xi_1)-a(x,\xi_2)| \leq c\Psi^\frac{p-1-\delta}{p}[(a(x,\xi_1)-a(x,\xi_2)\cdot(\xi_1-\xi_2)]^\frac{\delta}{p},$$
\item[$h_6$)] cyclical monotonicity: $\sum_{i=1}^k a(x,\xi_1)\cdot (\xi_{i+1}-\xi_i)\leq 0$, with $\xi_1=\xi_{k+1}$,
\item[$h_7$)] strict monotonicity: $c_1|\xi_1-\xi_2|^\gamma\Psi^{1-\frac{\gamma}{p}} \leq (a(x,\xi_1)-a(x,\xi_2))\cdot (\xi_1-\xi_2)$,
\end{itemize} 
where we have denoted  $\delta=\min\{p/2,p-1\}$, $\Psi=\Psi(x,\xi_1,\xi_2)=a(x,\xi_1)\cdot\xi_1 + a(x,\xi_2)\cdot\xi_2$ for all  $\xi_1,\xi_2\in\mathbb R^n$ and $x\in\Omega$, $\gamma=\max\{2,p\}$. See \cite[Section 3.4]{BC} for  a detailed discussion of these conditions.

For simplicity in the notation and the exposition, we have chosen to write the manuscript for the $p-$Laplacian.

The paper is organized as follows: in Section \ref{sec.prel} we introduce some preliminaries and notation, in Section \ref{sec2} we prove our key results on rapidly oscillating integrals, in Section \ref{sec3} we give the proof of our main results and Section \ref{sec4} is devoted to deal with a one-dimensional equation for a  weighted operator.

\section{Preliminaries} \label{sec.prel}
\subsection{Notation}
Given an open set $\Omega\subset \mathbb R^n$, $n\geq 1$ and $1< p<\infty$, we consider the Sobolev space
$$
W^{1,p}(\Omega)=\{u\in L^p(\Omega)\colon |\nabla u|\in L^p(\Omega) \}.
$$
This space is a separable and reflexive Banach space endowed with the norm 
$$
\|u\|_{W^{1,p}(\Omega)}=(\|u\|_{L^p(\Omega)}+ \|\nabla u\|_{L^p(\Omega)})^\frac1p
$$
where $\|u\|_{L^p (\Omega)} = \left(\int_\Omega |u|^p\,dx\right)^\frac1p$.

As usual, we denote by $W^{1,p}_0(\Omega)$ the closure of the of $C^\infty_c(\Omega)$ with respect to the norm of $W^{1,p}(\Omega)$. When $\partial\Omega\in C^1$, this space coincides with the set of functions with zero trace. Note that, by the Poincar\'e's inequality in $W^{1,p}_0(\Omega)$, $\| \nabla u\|_{L^p(\Omega)}$ becomes an equivalent norm in $W^{1,p}_0(\Omega)$.

\noindent Given $p>1$ we denote its \emph{conjugate exponent} by $p'=\frac{p}{p-1}$ ($p'=\infty$ if $p=1$).

\noindent We denote the unit cube in $\mathbb R^n$ by $Q$. For any $\varepsilon>0$, $Q_\varepsilon$ represents the scaled cube $\varepsilon Q$.

Given a bounded set $\Omega\subset \mathbb R^n$ we denote its \emph{diameter} by $\text{diam}(\Omega)$ and its \emph{$n-$dimensional Hausdorff measure} by $|\Omega|=\mathcal{H}^n(\Omega)$.

%\color{red}
%Let $\Omega\subset \mathbb R^n$ be an open set. We say that a sequence  $\{u_k\}_{k\in\mathbb N}$ of functions in $L^q(\Omega)$ with $1\leq q<\infty$ \emph{weakly converges} to a function $u\in L^q(\Omega)$ if
%$$
%\lim_{k\to\infty} \int_\Omega u_k v \,dx =\int_\Omega u v\,dx \qquad \forall v\in L^{q'}(\Omega)
%$$
%We say that a sequence of functions $\{u_k\}_{k\in\mathbb N}$ of functions in $L^\infty(\Omega)$ \emph{weakly* converges} to a function $u\in L^\infty(\Omega)$ if
%$$
%\lim_{k\to\infty} \int_\Omega u_k v \,dx =\int_\Omega u v\,dx \qquad \forall v\in L^1(\Omega).
%$$
%\normalcolor
\subsection{Eigenpairs of the $p-$Laplacian}
Given a bounded open set $\Omega\subset \mathbb R^n$, $n\geq 1$, with Lipschitz boundary, we consider the following eigenvalue problem with Dirichlet boundary condition
\begin{align} \label{eq.p}
\begin{cases}
-\Delta_p u +V|u|^{p-2}u = \lambda \rho |u|^{p-2}u &\text{ in } \Omega\\
u=0 & \text{ on } \partial\Omega.
\end{cases}
\end{align}
Here, for $1<p<\infty$ the weight functions $\rho$ and $V$ belong to $L^q(\Omega)$, with $q$ satisfying \eqref{rel1}--\eqref{rel2}, and are uniformly positive \text{a.e.} in $\Omega$.

We say that $\lambda\in \mathbb R$ is an \emph{eigenvalue} of \eqref{eq.p} if there exists a nontrivial function $u\in W^{1,p}_0(\Omega)$ such that for every $v\in W^{1,p}_0(\Omega)$
\begin{align} \label{debil}
\int_\Omega |\nabla u|^{p-2}\nabla u \cdot \nabla v\,dx +\int_\Omega V |u|^{p-2}uv\,dx = \lambda \int_\Omega \rho |u|^{p-2}uv\,dx.
\end{align}

The $k-$th variational eigenvalue of \eqref{eq.p} is given by
\begin{equation} \label{R.quot}
\lambda_k=\inf_{C\in \Gamma_k} \sup_{v\in C} \frac{\int_\Omega |\nabla v|^p\,dx + \int_\Omega V |v|^p\,dx}{\int_\Omega \rho|v|^p\,dx}
\end{equation}
where $\Gamma_k=\{C\subset W^{1,p}_0(\Omega)\colon C \text{ compact}, C=-C, \gamma(C)\geq k\}$ 
and $\gamma(C)$ is the \emph{Krasnoselskii genus} of the set $C$. In general, this sequence does not exhaust the spectrum unless $p=2$ or $n=1$. See \cite{L} for details.

Since $V,\rho\geq c_0>0$ \text{a.e.} in $\Omega$ and $V,\rho \in L^q(\Omega)$ with $q$ such that \eqref{rel1} holds,  by using H\"older's and Sobolev's inequalities it follows that \eqref{R.quot} is well defined  for any $v\in W^{1,p}_0(\Omega)$.

Under the same considerations on $\rho$ and $V$, we also consider the eigenvalue problem with Neumann boundary condition
\begin{align} \label{eq.n}
\begin{cases}
-\Delta_p u +V|u|^{p-2}u =\lambda  \rho |u|^{p-2}u &\text{ in } \Omega\\
|\nabla u|^{p-2}\nabla u\cdot \eta =0 & \text{ on } \partial\Omega,
\end{cases}
\end{align}
where $\eta$ is the outward normal to $\partial\Omega$.

In this case the $k-$th variational eigenvalue of \eqref{eq.n} is given by
$$
\lambda_k=\inf_{C\in \Gamma_k} \sup_{v\in C} \frac{\int_\Omega |\nabla v|^p\,dx + \int_\Omega V |v|^p\,dx}{\int_\Omega \rho|v|^p\,dx}
$$
where $\Gamma_k=\{C\subset W^{1,p}(\Omega)\colon C \text{ compact}, C=-C, \gamma(C)\geq k\}$.

\subsection{Some useful inequalities}
Throughout this article, we will frequently utilize several well-known inequalities: Sobolev's inequality, Poincaré's inequality, Trace inequality, and Morrey's inequality. For their precise statements and proofs, we refer, for instance, to \cite{Leoni}. Regarding the applicability of these inequalities, we define the following critical exponents:
\begin{align*}
p^*=
\begin{cases}
\frac{np}{n-p} &\text{ if } p<n\\
\infty &\text{ if } p\geq n
\end{cases}
\qquad \qquad 
p_*=
\begin{cases}
\frac{(n-1)p}{n-p} &\text{ if } p<n\\
\infty &\text{ if } p\geq n.
\end{cases}
\end{align*}

%\color{red}
%Given $u\in W^{1,p}(\Omega)$, we have:
%\begin{itemize}  
%\item[(i)] \emph{Sobolev's inequality}: there exists $C=C(p,\Omega)$ such that
%$$
%\|u\|_{L^r(\Omega)} \leq C \|u\|_{W^{1,p}(\Omega)}
%$$
%holds for  $1\leq r\leq \frac{np}{n-p}$ when $1\leq p<n$, and for $n\leq r <\infty$ when $p=n$.  
%\\
%When $u\in W^{1,p}_0(\Omega)$, the inequality
%$$
%\|u\|_{L^r(\Omega)} \leq C \|\nabla u\|_{L^p(\Omega)}
%$$
%holds  without assuming regularity on $\partial\Omega$.
%
%
%\item[(ii)] \emph{Poincar\'e's inequality:} there exists $C=C(p,\Omega)$ such that
%$$
%\|u-\bar u\|_{L^r(\Omega)} \leq C  \|\nabla u\|_{L^p(\Omega)}, 
%$$
%where $\bar u=\frac{1}{|\Omega|}\int_\Omega u\,dx$ is the average of $u$ in $\Omega$ and  $1\leq r\leq p^*$ with strict second inequality if $p=n$, that is, $1\leq r \leq \frac{np}{n-p}$ when $1\leq p<n$, $1\leq r <\infty$ when $p=n$ and $r=\infty$ when $p>n$.
%
%\item[(iii)] \emph{Trace inequality}: there exists $C=C(p,\Omega)$ such that
%$$
%\|u\|_{L^r(\partial\Omega)} \leq C \|u\|_{W^{1,p}(\Omega)}
%$$
%where $1\leq r\leq p_*$ with strict second inequality if $p=n$.
%
%\item[(iv)] \emph{Morrey's inequality:} when $p>n$ it holds that
%$$
%\|u\|_{C^{0,1-\frac{n}{p}}(\overline\Omega)}\leq C \|u\|_{W^{1,p}(\Omega)}
%$$
%where $C=C(p,\Omega)$.
%\end{itemize}
%
%
%
%We remark that in the one-dimensional case, given $u\in W^{1,p}(\Omega)$ with $\Omega=(a,b)\subset \mathbb R$ a bounded interval, it holds that $u\in C^{0,1-\frac{1}{p}}(\bar\Omega)$ for any $p>1$, and $\|u\|_{C^{0,1-\frac{n}{p}}(\overline \Omega)}\leq C \|u\|_{W^{1,p}(\Omega)}$.
%
% 
%\normalcolor

\section{Estimates on rapidly oscillating integrals} \label{sec2}

\noindent Along this section $\Omega\subset \mathbb R^n$ denotes an open bounded set with Lipschitz boundary. We denote  $\rho_\varepsilon(x):=\rho(\tfrac{x}{\varepsilon})$, being $\rho\in L^q(Q)$  a $Q-$periodic function. Then, $\rho_\varepsilon \rightharpoonup \bar \rho$ weakly (resp. weakly*) in $L^q(\Omega)$, when $1\leq q<\infty$ (resp. $q=\infty$),  with $\bar \rho$ the average of $\rho$ in $Q$.

The following result details Poincar\'e's inequality on the cube $Q_\varepsilon$, illustrating how the constant depends on $\varepsilon$. We  omit the proof since it is a direct consequence of applying the classical Poincar\'e's inequality with a stretching argument.

\begin{prop} \label{prop1}
For every $u\in W^{1,p}(Q_\varepsilon)$, there exists $c>0$ such that
$$
\|u-\bar u_\varepsilon\|_{L^r(Q_\varepsilon)} \leq c   \varepsilon^\beta \|\nabla u\|_{L^p(Q_\varepsilon)}
$$
where 
\begin{align*}
\beta= 
\begin{cases}
1+\frac{n}{r}-\frac{n}{p} &\text{ when } p \leq n \text{ and } 1\leq r \leq \frac{np}{n-p},\\
1-\frac{n}{p} &\text{ when } p> n \text{ and } r=\infty,
\end{cases}
\end{align*}
and $\bar u_\varepsilon$ is the average of $u$ in $Q_\varepsilon$.
\end{prop}

%\color{red}
%\begin{proof} 
%Let $u\in W^{1,r}(Q_\varepsilon)$. Assume first that $1\leq p \leq n$, and let $r$ be such that $1\leq r \leq \frac{np}{n-p}$ when $p<n$,  $1\leq r <\infty$ when $p=n$. We can assume that $\bar u_\varepsilon=0$. If we denote $u_\varepsilon(x)=u(\varepsilon x)$, we have that $u_\varepsilon \in W^{1,p}(Q)$ and by a change of variables we get
%\begin{align*}
%\|u\|_{L^r(Q_\varepsilon)}=\left(\int_{Q_\varepsilon} |u|^r\,dx \right)^\frac1r &= \left(\varepsilon^n \int_Q |u_\varepsilon|^r \,dx \right)^\frac1r \leq   \varepsilon^\frac{n}{r} c \left(\int_Q |\nabla u_\varepsilon|^p \,dx\right)^\frac{1}{p}\\
%&\leq 
%  c \varepsilon^\frac{n}{r}  \left( \varepsilon^{p-n}\int_{Q_\varepsilon} |\nabla u|^p \,dx\right)^\frac{1}{p}=c \varepsilon^{1+\frac{n}{r}-\frac{n}{p}} \|\nabla u\|_{L^p(Q_\varepsilon)}.
%\end{align*}
%When $p> n$, by inequality 11.26 in \cite{Leoni}, we have that
%$$
%|u(x)-\bar u_\varepsilon|\leq \frac{np}{p-n} \varepsilon^{1-\frac{n}{p}} \|\nabla u\|_{L^p(Q_\varepsilon)}
%$$
%for any $x\in Q_\varepsilon$, from where the desired inequality follows.
%\end{proof}
%\normalcolor

The following proposition states the behavior of rapidly oscillating integrals  related to the Dirichlet eigenvalue problem.
\begin{prop} \label{prop.dirichlet}
Let $p$ and $q$ satisfying condition \eqref{rel1}--\eqref{rel2}. Then
$$
\left|\int_\Omega(\rho_\varepsilon-\bar\rho) |u|^p\,dx \right|  \leq   C \varepsilon^\alpha \|\rho_\varepsilon - \bar \rho\|_{L^q(\Omega)}   \|\nabla u\|_{L^p(\Omega)}^p
$$
holds for all $u\in W^{1,p}_0(\Omega)$ and $\varepsilon\ll 1$, where 
\begin{align} \label{alfa}
\alpha=
\begin{cases}
p-\frac{n}{q} & \text{ when } q> \frac{n}{p} \text{ and } 1<p\leq n\\
1-\frac{n}{p} & \text{ when } q= 1 \text{ and } p>n
\end{cases}
\end{align}
and $C$ is a constant depending only of $p$  and $\Omega$.
\end{prop}

\begin{proof}
Let $u\in W^{1,p}_0(\Omega)$ and extend it by 0 outside $\Omega$. We observe that 
\begin{align*}
v:=|u|^p\in W^{1,r}_0(\Omega) \text{ with }
\begin{cases}
r\leq \frac{n}{n+1-p} &\text{ when } p\leq n\\
r\leq p &\text{ when } p>n.
\end{cases}
\end{align*}
Indeed, when $p\leq n$ and  $r= \frac{n}{n+1-p}$ by H\"older's and Sobolev's inequality we get
\begin{align} \label{cota1}
\begin{split}
\int_\Omega |\nabla v|^r\,dx &\leq p^r \int_\Omega |u|^{r(p-1)} |\nabla u|^r\,dx\\
&\leq p^r \left(\int_\Omega |u|^{r(p-1)(\frac{p}{r})' }\,dx \right)^{\frac{1}{(p/r)'}} \left(\int_\Omega |\nabla u|^p\,dx \right)^\frac{r}{p}\\
&\leq p^r \left(\int_\Omega |u|^{p^*}\,dx  \right)^\frac{p-r}{p}    \left(\int_\Omega |\nabla u|^p\,dx \right)^\frac{r}{p}\\
&\leq C \left(\int_\Omega  |\nabla u|^p \,dx \right)^{\frac{n(p-1)}{p(n+1-p)}+\frac{r}{p}}\leq C \left(\int_\Omega  |\nabla u|^p \,dx \right)^r,
\end{split}
\end{align}
where $C$ is a positive constant independent of $q$.
\\
When $p>n$ and $r\leq p$, by Morrey's estimate we get that
\begin{align} \label{cota2}
\begin{split}
\int_\Omega |\nabla v|^r\,dx &\leq p^r \int_\Omega |u|^{r(p-1)} |\nabla u|^r\,dx \leq p^r \|u\|_{L^\infty(\Omega)}^{r(p-1)} \int_\Omega |\nabla u|^r\,dx\\
&\leq p^r |\Omega|^\frac{p-r}{p} \|u\|_{L^\infty(\Omega)}^{r(p-1)} \left(\int_\Omega |\nabla u|^p\,dx\right)^\frac{r}{p}\leq C \left(\int_\Omega |\nabla u|^p\,dx\right)^{r}.
\end{split}
\end{align}

Denote by $I_\varepsilon$ the set of all $z\in \mathbb{Z}^n$ such that $Q_{z,\varepsilon}\cap \Omega\neq \emptyset$, $Q_{z,\varepsilon}:= z+\varepsilon Q$ and consider the piece-wise constant function $\bar v_\varepsilon$ given by
$$
\bar v_\varepsilon\big|_{Q_{z,\varepsilon}} :=\frac{1}{\varepsilon^n}\int_{Q_{z,\varepsilon}} v(y)\,dy \quad \text{ in }Q_{z,\varepsilon}.
$$
We denote $\displaystyle\Omega_0=\bigcup_{z\in I^\varepsilon} Q_{z,\varepsilon}$, then
$$
\int_\Omega (\rho_\varepsilon-\bar \rho) v\,dx= \int_{\Omega_0} (\rho_\varepsilon-\bar \rho) (v-\bar v_\varepsilon)\,dx  +  \int_{\Omega_0} (\rho_\varepsilon-\bar \rho)   \bar v_\varepsilon  \,dx:=(a)+(b). 
$$
Since $\rho$ is $Q-$periodic, we get
\begin{equation} \label{cotav0}
(b)=\sum_{z\in I^\varepsilon } \bar v_\varepsilon \big|_{Q_{z,\varepsilon}} \int_{Q_{z,\varepsilon}} (\rho_\varepsilon -\bar \rho)\,dx=0.
\end{equation}

Let us estimate $(a)$.  Let us first consider the case in which $p\leq n$ and therefore $v\in W^{1,r}(\Omega)$ with $r= \frac{n}{n+1-p}$. We observe that
$$
q'\leq r^* \iff n\leq pq,
$$
then by using Proposition \ref{prop1} we obtain that
\begin{align*}
\begin{split}
(a) &\leq \|\rho_\varepsilon - \bar \rho\|_{L^q(\Omega)} \| v- \bar v_\varepsilon\|_{L^{q'}(\Omega_0)} = \|\rho_\varepsilon - \bar \rho\|_{L^q(\Omega)}\sum_{z\in I^\varepsilon } \|v-\bar v_\varepsilon\|_{ L^{q'}(Q_{z,\varepsilon})} \\
&\leq   \|\rho_\varepsilon - \bar \rho\|_{L^q(\Omega)} \sum_{z\in I^\varepsilon } c\varepsilon^{1+\frac{n}{q'}-\frac{n}{r}} \|\nabla v\|_{L^r(Q_{z,\varepsilon})}\leq c\varepsilon^{p-\frac{n}{q} }   \|\rho_\varepsilon - \bar \rho\|_{L^q(\Omega)}  \|\nabla v\|_{L^r(\Omega)}.
\end{split}
\end{align*}
In light of \eqref{cota1}, the last expression gives 
\begin{equation} \label{cotav1}
(a) \leq c\varepsilon^{p-\frac{n}{q} }   \|\rho_\varepsilon - \bar \rho\|_{L^q(\Omega)}  \|\nabla u\|_{L^p(\Omega)}^p.
\end{equation}

When $n>p$, we have that $v\in W^{1,p}(\Omega)$. In this case, by using Proposition \ref{prop1} 
\begin{align*}
\begin{split}
(a) &\leq \|\rho_\varepsilon - \bar \rho\|_{L^1(\Omega)} \| v- \bar v_\varepsilon\|_{L^{\infty}(\Omega_0)} \\
&\leq \|\rho_\varepsilon - \bar \rho\|_{L^1(\Omega)}\sum_{z\in I^\varepsilon } \|v-\bar v_\varepsilon\|_{ L^\infty(Q_{z,\varepsilon})} \\
&\leq   \|\rho_\varepsilon - \bar \rho\|_{L^1(\Omega)} \sum_{z\in I^\varepsilon } c\varepsilon^{1-\frac{n}{p}} \|\nabla v\|_{L^p(Q_{z,\varepsilon})}\\
&\leq C\varepsilon^{1-\frac{n}{p} }   \|\rho_\varepsilon - \bar \rho\|_{L^1(\Omega)}  \|\nabla v\|_{L^p(\Omega)}.
\end{split}
\end{align*}
Then, due to \eqref{cota2}, the last expression gives
\begin{equation} \label{cotav2}
(a)\leq C\varepsilon^{1-\frac{n}{p} }   \|\rho_\varepsilon - \bar \rho\|_{L^1(\Omega)}  \|\nabla u\|_{L^p(\Omega)}^p.
\end{equation}
Finally, from  expression \eqref{cotav0}, \eqref{cotav1} and \eqref{cotav2} the proposition follows.
\end{proof}

The following estimate is essential to deal with the Neumann case.

For $\delta>0$ we denote by $G_\delta$ a tubular neighborhood of $\partial\Omega$, defined as $G_\delta=\{x\in\Omega\colon \text{\normalfont dist} (x,\partial\Omega)<\delta\}$.

\begin{prop} \label{prop2}
Let $p\geq 1$ and $r$ be such that $1\leq r \leq p_*$ with strict second inequality if $p\geq n$, that is, $1\leq r \leq \frac{(n-1)p}{n-p}$ when $p<n$ and $1\leq r <\infty$ when $p\geq n$.

Then there exists $\delta_0>0$ such that for every $\delta\in (0,\delta_0)$ and every $u\in W^{1,p}(\Omega)$ we have
$$
\|u\|_{L^r(G_\delta)} \leq C \delta^\frac1r \|u\|_{W^{1,p}(\Omega)}
$$
where $C$ is independent of $\delta$, $u$ and $r$.
\end{prop}

\begin{proof}
Take $1<p\leq n$ and  $G_\delta=\{x\in \Omega\colon \text{\normalfont dist}(x,\partial\Omega) <\delta\}$.  Since $\Omega$ domain with Lipschitz boundary,  there is a sufficiently small $\delta_0>0$ and a family of smooth surfaces $\{S_\tau\}_{\tau\in[0,\delta_0]}$ with the following properties:
\begin{itemize}
\item[i)] each $S_\tau=\partial \Omega_\tau$ for some  domain $\Omega_\tau\subset \Omega$,
\item[ii)] the domains are nested: $\Omega_{\tau'}\subset \Omega_\tau$ if $\tau'>\tau$ and $\Omega_0=\Omega$,
\item[iii)] for all $x\in S_\tau$ and $\tau\in [0,\delta_0]$ it holds that, for some positive constants $c_1,c_2$, 
$$c_1\tau \leq {\normalfont dist}(x,\partial\Omega) \leq c_2\tau.$$
\end{itemize}
Furthermore, $G_\tau \subset \Omega\setminus \Omega_\tau$.  When $r<\frac{(n-1)p}{n-p}$ (if $p<n$) or $r<\infty$ (if $p=n$), by the Sobolev trace Theorem we have
$$
\|u\|_{L^r(S_\tau)} \leq C \|u\|_{W^{1,p}(\Omega_\tau)} \leq C \|u\|_{W^{1,p}(\Omega)} \qquad \text{for }\tau\in[0,\delta_0],
$$
where the constant $C$  is independent of $\tau$. Integrating this inequality with respect to $\tau$ yields
$$
\|u\|_{L^r(G_\delta)}^r =\int_0^\delta \left(\int_{S_\tau} |u|^r\,dS \right)d\tau \leq \delta C^r \|u\|_{W^{1,p}(\Omega)}^r.
$$
This concludes the proof.
\end{proof}

The following proposition states the behavior of rapidly oscillating integrals  related to the Neumann eigenvalue problem.
\begin{prop} \label{prop.neumann}
Let  $1<p\leq n$ and $\frac{n-1}{p-1}< q \leq \infty$. Then
$$
\left|\int_\Omega(\rho_\varepsilon-\bar\rho) |u|^p\,dx \right|  \leq   C \varepsilon^{1-\frac1q} \|\rho_\varepsilon - \bar \rho\|_{L^q(\Omega)}  \|   u\|_{W^{1,p}(\Omega)}^p
$$
holds for all $u\in W^{1,p}(\Omega)$ and $\varepsilon\ll 1$,  where $C$ is a constant depending only of $p$  and $\Omega$.
\end{prop}

\begin{proof}
Let $u\in W^{1,p}(\Omega)$. We denote by $I_\varepsilon$ the set of all $z\in \mathbb{Z}^n$ such that $Q_{z,\varepsilon}:=\varepsilon (z+Q)\subset \Omega$ and consider the piece-wise constant function $\bar u_\varepsilon$ given by
$$
\bar u_\varepsilon\big|_{Q_{z,\varepsilon}} :=\frac{1}{\varepsilon^n}\int_{Q_{z,\varepsilon}} u(y)\,dy \quad \text{ in }Q_{z,\varepsilon}.
$$

We denote $\Omega_0=\bigcup_{z\in I^\varepsilon} Q_{z,\varepsilon}$ the union of all the cubes $Q_{z,\varepsilon}$ fully contained in $\Omega$ and $G=\Omega\setminus \Omega_0$. Then, we can split the integral as
\begin{align*}
\int_\Omega (\rho_\varepsilon-\bar \rho) |u|^p \,dx &= \int_{\Omega_0} (\rho_\varepsilon-\bar \rho) |u|^p  \,dx
 + \int_G (\rho_\varepsilon-\bar \rho) |u|^p \,dx:=(a)+(b).
\end{align*}

The term $(a)$ was already estimated in \eqref{cotav1} and \eqref{cotav2} in the Dirichlet case, but these estimates hold in fact for functions in $W^{1,p}(\Omega)$ (not only for functions in $W^{1,p}_0(\Omega)$). Then, it is obtained that
\begin{align*}
(a)\leq \left|\int_\Omega(\rho_\varepsilon-\bar\rho) |u|^p \,dx \right|  \leq   C \varepsilon^\alpha \|\rho_\varepsilon - \bar \rho\|_{L^q(\Omega)}   \|\nabla u\|_{L^{p}(\Omega)}^p, 
\end{align*}
where $\alpha$ is given in \eqref{alfa}.

We deal now with $(b)$. Observe that $G$ is contained in a $\delta-$neighborhood of $\partial\Omega$ with $\delta=c\varepsilon$ for some constant $c$. We apply H\"older's inequality with exponents $q$ and $q'$ and then Proposition  \ref{prop2} with $r=pq'$. 

In the case $p<n$ we have that $1\leq r \leq \frac{(n-1)p}{n-p}$ if and only if $\frac{n-1}{p-1}\leq q\leq \infty$.

In the case $p=n$, $1\leq r < \infty$ if and only if $1<q\leq \infty$.

Under these assumptions, we get
\begin{align*}
(b)&\leq \|\rho_\varepsilon -\bar \rho\|_{L^q(G)} \|u\|_{L^{pq'}(G)}^p\\
&\leq C \|\rho_\varepsilon -\bar \rho\|_{L^q(G)} \delta^\frac{1}{q'} \|u\|_{W^{1,p}(\Omega)}^p\leq C \varepsilon^{1-\frac{1}{q}}\|\rho_\varepsilon -\bar \rho\|_{L^q(\Omega)} \|u\|_{W^{1,p}(\Omega)}^p.
\end{align*}

Gathering the previous expressions, when $\frac{n-1}{p-1}< q \leq \infty$ and $p\leq n$ we get
\begin{align*}
\left|\int_\Omega (\rho_\varepsilon-\bar \rho) |u|^p\,dx \right| &\leq C  \|\rho_\varepsilon - \bar \rho\|_{L^q(\Omega)} \left( \varepsilon^{p-\frac{n}{q}} \|\nabla u\|_{L^p(\Omega)}^p +  \varepsilon^{1-\frac{1}{q}}  \|u\|_{W^{1,p}(\Omega)}^p \right)\\
&\leq 
C  \|\rho_\varepsilon - \bar \rho\|_{L^q(\Omega)} \|u\|_{W^{1,p}(\Omega)}^p  \max\{\varepsilon^{p-\frac{n}{q}}, \varepsilon^{1-\frac1q} \}.
\end{align*}

Finally, observe that $1-\frac1q \leq p-\frac{n}{q} \iff \frac{n-1}{p-1}\leq q$.  This concludes the proof.
\end{proof}

\section{Main results} \label{sec3}
When $\{ \rho_\varepsilon\}_{\varepsilon>0}$ and $\{V_\varepsilon\}_{\varepsilon>0}$ are general sequences in $L^q(\Omega)$ (not necessarily periodic), with $p$ and $q$ satisfying  \eqref{rel1}--\eqref{rel2}, our first result states the convergences of eigenvalues in the Dirichlet case as $\varepsilon\to 0$.

\begin{thm} \label{teo.d.sinorden}
Let $\lambda_{k,\varepsilon}$ and $\lambda_{k,0}$ be the $k-$th eigenvalue of problems \eqref{eq.p.eps} and \eqref{eq.p.0}, respectively. Then, 
$$
\lim_{\varepsilon\to 0} \lambda_{k,\varepsilon} = \lambda_{k,0}.
$$
\end{thm}

\begin{proof}
Let $\delta>0$ and let $G_\delta^k\subset W^{1,p}_0(\Omega)$ be a compact, symmetric set of genus $k$ such that 
$$
\lambda_{k,0} = \inf_{G\in \Gamma_k} \sup_{u\in G} \frac{\int_\Omega |\nabla u|^p + V_0 |u|^p \,dx}{\int_\Omega   \rho_0 |u|^p\,dx} = \sup_{u\in G_\delta^k} \frac{\int_\Omega |\nabla u|^p + V_0|u|^p\,dx}{\int_\Omega   \rho_0 |u|^p\,dx} + O(\delta).
$$
The set $G_\delta^k$ is admissible in the variational characterization of $\lambda_{k,\varepsilon}$, then
\begin{align*}
\lambda_{k,\varepsilon} &\leq \sup_{u\in G_\delta^k} \frac{\int_\Omega |\nabla u|^p  +   V_\varepsilon |u|^p\,dx}{\int_\Omega  \rho_\varepsilon |u|^p\,dx}\\ 
&=  \sup_{u\in G_\delta^k} \left( \frac{\int_\Omega |\nabla u|^p +  V_0 |u|^p \,dx}{\int_\Omega   \rho_0 |u|^p\,dx}  +  \frac{\int_\Omega (V_\varepsilon -  V_0) |u|^p \,dx}{\int_\Omega  \rho_0 |u|^p\,dx} \right)
\frac{\int_\Omega   \rho_0 |u|^p\,dx}{\int_\Omega  \rho_\varepsilon |u|^p\,dx} . 
\end{align*}

When $1\leq q<\infty$, by assumption  we have that $\rho_\varepsilon \rightharpoonup \rho_0$ and $V_\varepsilon \rightharpoonup V_0$ weakly in $L^q(\Omega)$, this means that
\begin{equation} \label{conv.debil}
\lim_{\varepsilon\to 0}\int_\Omega \rho_\varepsilon \varphi \,dx = \int_\Omega \rho_0 \varphi\,dx, \quad \lim_{\varepsilon\to 0}\int_\Omega V_\varepsilon \varphi \,dx = \int_\Omega V_0 \varphi\,dx \qquad \forall \varphi \in L^{q'}(\Omega).
\end{equation}
Since $pq'\leq p^*$  for any $p>1$ due to \eqref{rel1}--\eqref{rel2}, we have by Sobolev's inequality that $|u|^p\in L^{q'}(\Omega)$, and then \eqref{conv.debil} gives that
$$
 \frac{\int_\Omega   \rho_0 |u|^p\,dx}{\int_\Omega  \rho_\varepsilon |u|^p\,dx} =1+ o(1), \quad \int_\Omega (V_\varepsilon -   V_0)|u|^p\,dx = o(1) \qquad \text{ as } \varepsilon \to 0.
$$
Therefore,  for every $u\in G_\delta^k$ we have that
$$
\lim_{\varepsilon\to0}\lambda_{k,\varepsilon}\leq \lim_{\varepsilon\to0}\sup_{u\in G_\delta^k} \left( \frac{\int_\Omega |\nabla u|^p +   V_0 |u|^p \,dx + o(1)  }{\int_\Omega    \rho_0 |u|^p\,dx } \right)(1+o(1)).
$$
Therefore, letting $\delta\to 0$, we obtain that $\lim_{\varepsilon\to 0} \lambda_{k,\varepsilon} \leq \lambda_{k,0}$. The case $q=\infty$ is similar.
 
By interchanging the roles of \(\lambda_{k,\varepsilon}\) and \(\lambda_{k,0}\), a similar argument yields \(\lim_{\varepsilon \to 0} \lambda_{k,\varepsilon} \geq \lambda_{k,0}\), thereby completing the proof.
\end{proof}

Assume now that  $\rho,V\in L^q(Q)$ are $Q-$periodic functions, where $q$ satisfies condition \eqref{rel1}--\eqref{rel2}. Then, we denote the sequence of rapidly oscillating functions $\{\rho_\varepsilon\}_{\varepsilon<1}$ and $\{V_\varepsilon\}_{\varepsilon<1}$ by $\rho_\varepsilon(x):=\rho(\tfrac{x}{\varepsilon})$ and $V_\varepsilon(x):=V(\tfrac{x}{\varepsilon})$ and by $\bar \rho$ and $\bar V$ their limit averages.

The following lemma gives a bound of the difference between a rapidly oscillating function and its limit.

\begin{lema} \label{cota.rho}
Let $1\leq q\leq \infty$. Given a bounded open set $\Omega\subset \mathbb R^n$,  we have that
$$
\|\rho_\varepsilon - \bar \rho\|_{L^q(\Omega)} \leq 
\|\rho\|_{L^q(Q)}\left( \mathop{diam}(\Omega)^\frac{n}{q} + |\Omega|^\frac1q \right).
$$
\end{lema}

\begin{proof}

Since $\Omega$ is bounded, we assume up to a translation  that $\Omega\subset Q_d$ for some $d>1$, where $Q_d$ is a $n-$cube of side $d=\mathop{diam}(\Omega)$. Denote $I_\varepsilon$ the set of all $z\in \mathbb{Z}^n$ such that $Q_d\cap Q_{z,\varepsilon} \neq \emptyset$, where $Q_{z,\varepsilon}:=z+ \varepsilon Q$. Then, since $\rho$ is a $Q$ periodic function, we have
\begin{align*}
\int_\Omega |\rho_\varepsilon(x)|^q \,dx &\leq  \sum_{I_\varepsilon}\int_{Q_{z,\varepsilon}} |\rho_\varepsilon(x)|^q\,dx\\
&=  \varepsilon^n \sum_{I_\varepsilon}\int_{Q} |\rho(x)|^q\,dx \leq \|\rho\|_{L^q(Q)}^q \varepsilon^n \# I_\varepsilon\\
&\leq
[d/\varepsilon]^n \varepsilon^n  \|\rho\|_{L^q(Q)}^q \leq
d^n  \|\rho\|_{L^q(Q)}^q
\end{align*}
since there are $[d/\varepsilon]^n$ cubes of side $\varepsilon$ contained in $Q_d$, being $[x]$ the integer part of $x\in\mathbb R$. Then, since 
$$
\bar \rho |\Omega|^\frac1q =  \frac{|\Omega|^\frac1q}{|Q|} \int_Q \rho(x)\,dx \leq 
\frac{|\Omega|^\frac1q}{|Q|} \|\rho\|_{L^q(Q)} |Q|^{1-\frac1q},
$$
we get the following bound:
\begin{align*}
\|\rho_\varepsilon - \bar \rho\|_{L^q(\Omega)} &\leq \|\rho_\varepsilon\|_{L^q(\Omega)} + \|\bar\rho\|_{L^q(\Omega)} \leq d^\frac{n}{q}  \|\rho\|_{L^q(Q)} + \bar\rho |\Omega|^\frac1q \\
&\leq    \|\rho\|_{L^q(Q)}\left( d^\frac{n}{q} + |\Omega|^\frac1q \right),
\end{align*}
which concludes the proof.
\end{proof}

The rate of convergence obtained in Theorem \ref{teo.d.sinorden} can be estimated under the periodicity condition of the weight sequences.

\begin{thm} \label{teo.d}
  Let $\lambda_{k,\varepsilon}$ and $\lambda_{k,0}$ the $k-$th eigenvalue of problems \eqref{eq.p.eps} and \eqref{eq.p.0}, respectively. Then there are (computable) constants $C$ and $\textbf{C}$ depending only of $p$, $q$, $\rho$, $V$ and $\Omega$ such that
\begin{align*}
|\lambda_{k,0}-\lambda_{k,\varepsilon}| &\leq  
C  \varepsilon^\alpha \left(\lambda_{k,0} \|V_\varepsilon - \bar V\|_{L^q(\Omega)}  +  \lambda_{k,0}^2 \|\rho_\varepsilon - \bar \rho\|_{L^q(\Omega)}  \right)
\end{align*}
holds for $0<\varepsilon<(2\textbf{C} \lambda_{k,0})^{-\frac{1}{\alpha}}$, where
$$
\alpha=
\begin{cases}
p-\frac{n}{q} & \text{ when } q> \frac{n}{p} \text{ and } 1<p\leq n\\
1-\frac{n}{p} & \text{ when } q=1 \text{ and } p>n.
\end{cases}
$$
\end{thm}

\begin{proof}
Let $\delta>0$ and let $G_\delta^k\subset W^{1,p}_0(\Omega)$ be a compact, symmetric set of genus $k$ such that 
$$
\lambda_{k,0} = \inf_{G\in \Gamma_k} \sup_{u\in G} \frac{\int_\Omega |\nabla u|^p + \bar V |u|^p \,dx}{\int_\Omega \bar \rho |u|^p\,dx} 
= \sup_{u\in G_\delta^k} \frac{\int_\Omega |\nabla u|^p + \bar V |u|^p\,dx}{\int_\Omega \bar \rho |u|^p\,dx} + O(\delta).
$$
The set $G_\delta^k$ is admissible in the variational characterization of $\lambda_{k,\varepsilon}$, then
\begin{align*}
\lambda_{k,\varepsilon} &\leq \sup_{u\in G_\delta^k} \frac{\int_\Omega |\nabla u|^p  +   V_\varepsilon |u|^p\,dx}{\int_\Omega  \rho_\varepsilon |u|^p\,dx}\\ 
&=  \sup_{u\in G_\delta^k} \left( \frac{\int_\Omega |\nabla u|^p + \bar V |u|^p \,dx}{\int_\Omega  \bar \rho |u|^p\,dx}  +  \frac{\int_\Omega (V_\varepsilon - \bar V) |u|^p \,dx}{\int_\Omega  \bar \rho |u|^p\,dx} \right)
\frac{\int_\Omega  \bar \rho |u|^p\,dx}{\int_\Omega  \rho_\varepsilon |u|^p\,dx} . 
\end{align*}
Observe that, for every $u\in G_\delta^k$ we have that
\begin{equation} \label{eqd.1}
\frac{\int_\Omega |\nabla u|^p + \bar V |u|^p \,dx}{\int_\Omega  \bar \rho |u|^p  \,dx} \leq  \sup_{u\in G_\delta^k} \frac{\int_\Omega |\nabla u|^p + \bar V |u|^p\,dx}{\int_\Omega  \bar \rho |u|^p\,dx} = \lambda_{k,0}+O(\delta). 
\end{equation}
Proposition \ref{prop.dirichlet} and \eqref{eqd.1} yield
\begin{align} \label{eqd.2}
\begin{split}
\frac{\int_\Omega (V_\varepsilon - \bar V) |u|^p \,dx}{\int_\Omega  \bar \rho |u|^p\,dx} &\leq  C_{V_\varepsilon}
 \varepsilon^\alpha    \frac{\int_\Omega  |\nabla u|^p \,dx}{\int_\Omega  \bar \rho |u|^p\,dx}\leq
 C_{V_\varepsilon} \varepsilon^\alpha   (\lambda_{k,0}+O(\delta)),
\end{split}
\end{align}
where we have denoted $C_{V_\varepsilon}:=C\|V_\varepsilon - \bar V\|_{L^q(\Omega)}$.

Since $u\in G_\delta^k\subset W^{1,p}_0(\Omega)$, by using Proposition \ref{prop.dirichlet}, \eqref{eqd.1} and Lemma \ref{cota.rho} we obtain that
\begin{align*}
 \frac{\int_\Omega  \bar \rho |u|^p\,dx}{\int_\Omega  \rho_\varepsilon |u|^p\,dx} &\leq 1 + C_{\rho_\varepsilon} \varepsilon^\alpha  \frac{\int_\Omega |\nabla u|^p\,dx}{\int_\Omega \rho_\varepsilon |u|^p\,dx}\\
&= 1 + C_{\rho_\varepsilon} \varepsilon^\alpha   \frac{\int_\Omega |\nabla u|^p\,dx}{\int_\Omega \bar \rho |u|^p\,dx} \frac{\int_\Omega \bar \rho |u|^p\,dx}{\int_\Omega \rho_\varepsilon |u|^p\,dx}\\
&\leq 1 + C_\rho \varepsilon^\alpha (\lambda_{k,0}+O(\delta)) \frac{\int_\Omega \bar \rho |u|^p\,dx}{\int_\Omega \rho_\varepsilon |u|^p\,dx}
\end{align*}
where $C_\rho$ is the bound of $C_{\rho_\varepsilon}$ (independent of $\varepsilon$ and $k$) given in  Lemma \ref{cota.rho}. Then, for $\varepsilon$ small enough, i.e., $\varepsilon<(C_\rho ( \lambda_{k,0}+O(\delta))^{-\frac{1}{\alpha}}$, this gives that
$$
\frac{\int_\Omega  \bar \rho |u|^p\,dx}{\int_\Omega  \rho_\varepsilon |u|^p\,dx}\leq \frac{1}{1- C_\rho \varepsilon^\alpha  (\lambda_{k,0}+O(\delta))}.
$$
Hence, the last two inequalities yield
\begin{equation} \label{eqd.3}
\frac{\int_\Omega  \bar \rho |u|^p\,dx}{\int_\Omega  \rho_\varepsilon |u|^p\,dx}\leq 1+ \varepsilon^\alpha \frac{ C_{\rho_\varepsilon} (\lambda_{k,0}+O(\delta))}{1- C_\rho\varepsilon^\alpha  (\lambda_{k,0}+O(\delta)) }.
\end{equation}
Gathering \eqref{eqd.1}, \eqref{eqd.2} and \eqref{eqd.3} and letting $\delta\to 0$, we get
\begin{align*}
\lambda_{k,\varepsilon}&\leq \lambda_{k,0}(1 +C_{V_\varepsilon} \varepsilon^\alpha  ) \left( 1+ \varepsilon^\alpha\frac{ C_{\rho_\varepsilon}   \lambda_{k,0}}{1- C_\rho \varepsilon^\alpha   \lambda_{k,0} } \right)
\end{align*}
that is, 
\begin{equation} \label{eq.r.1}
\lambda_{k,\varepsilon}-\lambda_{k,0} \leq C_{V_\varepsilon} \lambda_{k,0} \varepsilon^\alpha +  \frac{2 C_{\rho_\varepsilon} (1+C_{V_\varepsilon}) \lambda_{k,0}^2}{1- C_\rho  \varepsilon^\alpha \lambda_{k,0}}\varepsilon^\alpha
\end{equation}
and it holds for $\varepsilon< \min\{1,( C_\rho\lambda_{k,0})^{-\frac{1}{\alpha}}\}$.

We will obtain now an inequality similar to \eqref{eq.r.1} which bounds the difference $\lambda_{k,0}-\lambda_{k,\varepsilon}$. For this end, we observe first that   $\lambda_{k,\varepsilon}$ can be bounded independently of $\varepsilon$. For instance, if we take  $\varepsilon=  (2  C_\rho \lambda_{k,0})^{-\frac{1}{\alpha}}$ we can obtain from \eqref{eq.r.1} the following:
\begin{align*}
\lambda_{k,\varepsilon} &\leq \lambda_{k,0} +  C_V \lambda_{k,0}\varepsilon^\alpha + \frac{2C_\rho (1+C_V)\lambda_{k,0}^2}{1-C_\rho \varepsilon^\alpha \lambda_{k,0}}\varepsilon^\alpha\leq \lambda_{k,0} +  \frac{  C_V}{2C_\rho}  +    2 (1+C_V) \lambda_{k,0}
\end{align*}
(where the constant $C_V$ is defined as  $C_\rho$ and it is independent of $\varepsilon$), then
\begin{equation} \label{desig1}
\lambda_{k,\varepsilon} \leq   \left(3+ \frac{ C_V }{2 C_\rho}\frac{1}{\lambda_{1,0}} +2 C_V \right) \lambda_{k,0} = \tilde C \lambda_{k,0},
\end{equation}
where $\tilde C=\tilde C(p,V,\rho,\Omega)>1$ is independent of $\varepsilon$ and $k$.

Using \eqref{desig1} and a reasoning similar to the one leading to  \eqref{eq.r.1}, we can obtain that
\begin{align} \label{eq.r.2}
\begin{split}
\lambda_{k,0}-\lambda_{k,\varepsilon} \leq\tilde C C_{V_\varepsilon} \lambda_{k,0} \varepsilon^\alpha +   \frac{2 \tilde C^2  C_{\rho_\varepsilon} (1+C_{V_\varepsilon}) \lambda_{k,0}^2}{1- C_\rho \tilde C  \varepsilon^\alpha \lambda_{k,0}} \varepsilon^\alpha
\end{split}
\end{align}
for $\varepsilon <(\bar C_\rho \tilde C \lambda_{k,0})^{-\frac{1}{\alpha}}$. From \eqref{eq.r.1} and \eqref{eq.r.2} we obtain that
\begin{align*}
|\lambda_{k,0}-\lambda_{k,\varepsilon}| &\leq  
\tilde C C_{V_\varepsilon} \lambda_{k,0} \varepsilon^\alpha +  \frac{2 \tilde C^2 C_{\rho_\varepsilon} (1+C_{V_\varepsilon}) \lambda_{k,0}^2}{1-C_\rho \tilde C \varepsilon^\alpha \lambda_{k,0}}\varepsilon^\alpha 
\end{align*}
for $\varepsilon < \min\{1,(\bar C_\rho \tilde C \lambda_{k,0})^{-\frac{1}{\alpha}}\}$.

Finally, by using again Lemma \ref{cota.rho}, for some constants $C$ and $\textbf{C}$  depending only of $p$,  $V$, $\rho$ and $\Omega)$ we can rewrite the previous relation as
\begin{align*}
|\lambda_{k,0}-\lambda_{k,\varepsilon}| &\leq  
C  \varepsilon^\alpha \|V_\varepsilon - \bar V\|_{L^q(\Omega)} \lambda_{k,0} + C  \|\rho_\varepsilon - \bar \rho\|_{L^q(\Omega)}  \varepsilon^\alpha \frac{ \lambda_{k,0}^2}{1- \textbf{C} \varepsilon^\alpha \lambda_{k,0}}
\end{align*}
and it holds for instance when $\varepsilon < (2\textbf{C} \lambda_{k,0})^{-\frac{1}{\alpha}}$. This completes the proof.
\end{proof}

By using the Weyl's estimates for eigenvalues, we can obtain the dependence on $k$ in the rates of convergence.

\begin{cor} \label{cor.d}
Under the assumptions of Theorem \ref{teo.d}   we have that
\begin{align*}
|\lambda_{k,0}-\lambda_{k,\varepsilon}| \leq C \varepsilon^\alpha \left( \|V \|_{L^q(\Omega)} k^\frac{p}{n} +     \|\rho \|_{L^q(\Omega)}   k^\frac{2p}{n} \right)
\end{align*}
holds for $0<\varepsilon<(2\textbf{C} \lambda_{k,0})^{-\frac{1}{\alpha}}$, where
$$
\alpha=
\begin{cases}
p-\frac{n}{q} & \text{ when } q> \frac{n}{p} \text{ and } 1<p\leq n\\
1-\frac{n}{p} & \text{ when } q= 1 \text{ and } p>n.
\end{cases}
$$
and $C = C (p,q, \rho,V,\Omega)$  is a positive constant independent of $k$ and $\varepsilon$.
\end{cor}
\begin{proof}

Observe that $\lambda_{k,0}$ is an eigenvalue corresponding to an equation with constant weights. In particular, given $u\in W^{1,p}_0(\Omega)$, we have that
$$
\frac{\int_\Omega |\nabla u|^p\,dx + \bar V \int_\Omega |u|^p\,dx}{\bar \rho \int_\Omega |u|^p\,dx} = \frac{1}{\bar\rho}\frac{\int_\Omega |\nabla u|^p\,dx}{\int_\Omega |u|^p\,dx} + \frac{\bar V}{\bar \rho},
$$
which, in light of the variational characterization of $\lambda_{k,0}$, it gives that
$$
\lambda_{k,0}\leq \frac{\mu_k}{\bar \rho}  + \frac{\bar V}{\bar \rho}, 
$$
being $\mu_k$ in the $k-$th variational eigenvalue of the  $p-$Laplacian  with Dirichlet boundary condition. Bounds of $\mu_k$ are well-known: in \cite{Fried} it is proved that
$$
\mu_k \leq \tilde \mu_1 k^\frac{p}{n} |\Omega|^{-\frac{p}{n}},
$$
where $\tilde \mu_1$ is the first eigenvalue of the Dirichlet $p-$Laplacian in the unit cube $Q$.  
\
The previous arguments yield $\lambda_{k,0} \leq C k^\frac{p}{n}$ with $C=C(\rho,V,p, \Omega)$.

Now, let $\textbf{C}$ be the constant given in Theorem \ref{teo.d} and consider $\varepsilon<(2\textbf{C} \lambda_{k,0})
^{-\frac{1}{\alpha}}$, then by Theorem \ref{teo.d} and Lemma \ref{cota.rho} we conclude that
\begin{align*}
|\lambda_{k,0}-\lambda_{k,\varepsilon}| &\leq  
  C \varepsilon^\alpha \|V  \|_{L^q(\Omega)} \lambda_{k,0} +    C  \|\rho\|_{L^q(\Omega)}  \varepsilon^\alpha \lambda_{k,0}^2\\
&\leq 
\bar C \varepsilon^\alpha \left( \|V\|_{L^q(\Omega)} k^\frac{p}{n} +   \|\rho\|_{L^q(\Omega)}  k^\frac{2p}{n} \right)
\end{align*}
where $\bar C = \bar C (\rho,V,p, \Omega)$. This gives the result.
\end{proof}

The analogous results for the Neumann case can be stated as follows:
\begin{thm} \label{teo.n}
Let $p$ and $q$ satisfying \eqref{rel1}--\eqref{rel2} and let $\lambda_{k,\varepsilon}$ and $\lambda_{k,0}$ be the $k-$th eigenvalue of problems \eqref{eq.n.eps} and \eqref{eq.n.0}, respectively. Then, 
$$
\lim_{\varepsilon\to 0} \lambda_{k,\varepsilon} = \lambda_{k,0}.
$$
Assume now that $1<p\leq n$ and $\frac{n-1}{p-1}<q\leq \infty$. When $\rho_\varepsilon:=\rho(\tfrac{x}{\varepsilon})$ and $V_\varepsilon:=V(\tfrac{x}{\varepsilon})$ are given in terms of $Q-$periodic functions $\rho,V$, then
\begin{align*}
|\lambda_{k,0}-\lambda_{k,\varepsilon}| &\leq  
C  \varepsilon^{1-\frac1q} \left( \|V_\varepsilon - \bar V\|_{L^q(\Omega)} \lambda_{k,0} +  2 \lambda_{k,0}^2 \|\rho_\varepsilon - \bar \rho\|_{L^q(\Omega)}   \right)
\end{align*}
holds for $0<\varepsilon< (2\textbf{C} \lambda_{k,0})^\frac{q}{1-q}$, where $C$   and $\textbf{C}$ depend only of $p$,   $\rho$, $V$ and $\Omega$ and are independent of $k$ and $\varepsilon$. In particular,   for $0<\varepsilon \leq (2\textbf{C} \lambda_{k,0})^\frac{q}{1-q}$ it holds that
\begin{align*}
|\lambda_{k,0}-\lambda_{k,\varepsilon}| \leq C \varepsilon^{1-\frac1q} \left( \|V \|_{L^q(\Omega)} k^\frac{p}{n} +     \|\rho \|_{L^q(\Omega)}   k^\frac{2p}{n} \right).
\end{align*}
\end{thm}

\begin{proof}
The proof of the convergence of the sequence $\lambda_{k,\varepsilon}$ to $\lambda_{k,0}$ as $\varepsilon\to 0$ follows analogously as in the proof of Theorem \ref{teo.d.sinorden}. The rate of convergence is obtained as in the proof of Theorem \ref{teo.d}   by using the variational characterization of Neumann eigenvalues and by applying  Proposition \ref{prop.neumann} instead of Proposition \ref{prop.dirichlet}. The last estimate follows as in the proof of Corollary \ref{cor.d} since in the case of problems with constant weights, Neumann eigenvalues  are lower than Dirichlet eigenvalues.
\end{proof}

\section{The one-dimensional case} \label{sec4}
In this section we consider a one-dimensional eigenvalue problem with Dirichlet boundary condition for a weighted $p-$Laplacian. Here we take into account an operator involving rapidly oscillating functions $a_\varepsilon$ and $\rho_\varepsilon$ given by $a_\varepsilon(x):=a(\tfrac{x}{\varepsilon})$ and $\rho_\varepsilon(x):=\rho(\tfrac{x}{\varepsilon})$, $\varepsilon>0$, given in terms of positive $1-$periodic functions $a\in L^\infty(\mathbb R)$ and $\rho \in L^q(I)$, $q\geq 1$, being $I$ the unit interval: 
\begin{align} \label{eq.p.a}
\begin{cases}
-(a(\tfrac{x}{\varepsilon})| u_\varepsilon|^{p-2}u_\varepsilon ')'=\lambda_\varepsilon \rho(\tfrac{x}{\varepsilon}) |u_\varepsilon|^{p-2}u_\varepsilon &\quad \text{ in } (0,1)\\
u_\varepsilon(0)=u_\varepsilon(1)=0.
\end{cases}
\end{align}
Due to the one-dimensional  H\"older's and Morrey's inequalities, this equation is well defined for any $p>1$ and $q\geq 1$.

The limit equation of \eqref{eq.p.a} as $\varepsilon\to 0$  is well-known. Indeed, when $\varepsilon\to 0$ it is obtained (see for instance \cite{A})
\begin{align} \label{eq.p.a.lim}
\begin{cases}
-(a^* | u_0|^{p-2}u_0 ')'=\lambda_0 \bar \rho |u_0|^{p-2}u_0 &\quad \text{ in } (0,1)\\
u_0(0)=u_0(1)=0.
\end{cases}
\end{align}
where $a^* = \left(\int_0^1 a(t)^{-\frac{1}{p-1}}\,dt \right)^{-(p-1)}$ and $\bar \rho=\int_0^1 \rho(t)\,dt$.

Unlike \eqref{eq.p.eps}, the operator considered here is a weighted $p-$Laplacian. To study the convergence of eigenvalues and its rate, we will apply a suitable change of variables to reduce \eqref{eq.p.a} to an equation involving the usual $p-$Laplacian and a rapidly oscillating function on the right-hand side of the equation. We follow the arguments in \cite{FPS0}.

\begin{prop} \label{cv}
The eigenvalue problem \eqref{eq.p.a} can be reduced to 
\begin{align}    \label{eqqq}
\begin{cases}
-(|  w_\delta'|^{p-2}w_\delta')'  =\mu_\delta g(\tfrac{z}{\delta}) |w_\delta|^{p-2}w_\delta &\quad z\in(0,1)\\
w_\delta(0)=w_\delta(1)=0.
\end{cases}
\end{align}	
where $\delta>0$ and $g\in L^q(I)$, $q\geq1$,  is a positive $1$-periodic function given  in terms of $\rho$ and $a$.
\end{prop}

\begin{proof}
We define the rapidly oscillating function
$$
P_\varepsilon(x):=\int_0^x a_\varepsilon(s)^{-\frac{1}{p-1}}\,ds = \varepsilon\int_0^\frac{x}{\varepsilon} a(s)^{-\frac{1}{p-1}} = \varepsilon P(\tfrac{x}{\varepsilon})
$$
and perform the change of variables $(x,u)\mapsto (y,v)$, where $y=P_\varepsilon(x)$, $v(y)=u(x)$. Then
\begin{align}   \label{eq.n1.1im}
\begin{cases}
-(|\dot v|^{p-2} \dot v)\dot\, =\lambda_\varepsilon Q_\varepsilon(y) |v|^{p-2}v &\quad y\in (0,L_\varepsilon)\\
v(0)=v(L_\varepsilon)=0.
\end{cases}
\end{align}	
where $\dot\,=\tfrac{d}{dy}$ and 
$$
L_\varepsilon= \int_0^1 a_\varepsilon(s)^{-\frac{1}{p-1}}\,ds, \quad  Q_\varepsilon(y)=a_\varepsilon(x)^\frac{1}{p-1} \rho_\varepsilon(x)=a(P^{-1}(\tfrac{y}{\varepsilon}))^\frac{1}{p-1}\rho(P^{-1}(\tfrac{y}{\varepsilon}))=Q(\tfrac{y}{\varepsilon}).
$$
Observe that $Q$ is a $L-$periodic function and $L_\varepsilon \to L=\overline{a^{\frac{-1}{p-1}}}$ as $\varepsilon \to 0$.
Moreover, 
\begin{align} \label{condL}
|L_\varepsilon-L|\leq 
\begin{cases}
\varepsilon L &\text{ if } \frac{1}{\varepsilon} \not\in \mathbb N\\
0 &\text{ if } \frac{1}{\varepsilon} \in \mathbb N.
\end{cases}
\end{align}
Now, we rescale \eqref{eq.n1.1im} to the unit interval. For that purpose define $w(z)=v(L_\varepsilon z)$, $z\in I$, and this gives
\begin{align*}   
\begin{cases}
-(|\dot w|^{p-2}w)\dot\, =L_\varepsilon^p\lambda_\varepsilon Q_\varepsilon(L_\varepsilon z) |w|^{p-2}w &\quad z\in (0,1)\\
w(0)=w(1)=0.
\end{cases}
\end{align*}	
If we denote $\delta=\frac{\varepsilon L}{L_\varepsilon}$, $\mu_\delta=L^p_\varepsilon \lambda_\varepsilon$ and $g(z)=Q(Lz)$, we get that $g$ is a $1-$periodic function and $w$ verifies
\begin{align*}   
\begin{cases}
-(|\dot w|^{p-2}w)\dot\, =\mu_\delta g(\tfrac{z}{\delta}) |w|^{p-2}w &\quad z\in (0,1),\\
w(0)=w(1)=0.
\end{cases}
\end{align*}	
Moreover, since $\rho\in L^q(I)$ is a $1-$periodic function and $a\in L^\infty$, we have  that $g\in L^q(I)$. This concludes the proof.
\end{proof}

In order to deal with the oscillations of \eqref{eqqq} we provide for an elemental proof of the one-dimensional version of Proposition \ref{prop.dirichlet}.

\begin{prop} \label{prop1d}
Let $p>1$ and $q\geq 1$, then for any $\varepsilon>0$ 
$$
\left|\int_0^1 (\rho(\tfrac{x}{\varepsilon})-\bar\rho) |u|^p\,dx \right| \leq C \varepsilon \|\rho-\bar \rho\|_{L^q(I)} \|u'\|_{L^{p}(I)}^p
$$
holds for $u\in W^{1,p}_0(I)$ where $C=C(p)$.
\end{prop}

\begin{proof}
Given $u\in W^{1,p}_0(I)$ it follows that $u\in C^{0,1-\frac1p}(I)$ and therefore $u$ is bounded in $I$. Let $v:=|u|^p$ and  $r\geq 1$, then
\begin{align*} 
\begin{split}
\int_0^1 |v'|^r\,dx &\leq p^r \int_0^1 |u|^{r(p-1)} |u'|^r\,dx\\
&\leq p^r \left(\int_0^1 |u|^{r(p-1)(\frac{p}{r})' }\,dx \right)^{\frac{1}{(p/r)'}} \left(\int_0^1 |u'|^p\,dx \right)^\frac{r}{p}\\
&\leq p^r \|u\|_{L^\infty(I)}^\frac{r(p-1)}{p} \|u'\|_{L^p(I)}^r \leq C \|u'\|_{L^p(I)}^{rp}.
\end{split}
\end{align*}
Therefore $v\in W^{1,r}_0(I)$ for any $r\geq 1$ and $\|v'\|_{L^r(I)}\leq C \|u'\|_{L^p(I)}^p$.

Consider the $1-$periodic function $R(x)=\int_0^x (\rho(t)-\bar \rho)\,dt$ and for $\varepsilon>0$ denote $R_\varepsilon(x)=R(\tfrac{x}{\varepsilon})$. Then
$$
\int_0^1 (\rho(\tfrac{x}{\varepsilon})-\bar\rho) v\,dx = \varepsilon \int_0^1 R_\varepsilon'(\tfrac{x}{\varepsilon})v\,dx = - \int_0^1 R_\varepsilon(x)v'\,dx,
$$
from where we get
\begin{align*}
\left|\int_0^1 (\rho(\tfrac{x}{\varepsilon})-\bar\rho) v\,dx \right| &\leq  \varepsilon \left| \int_0^1 R_\varepsilon(x) v'\,dx\right| \leq \varepsilon \|\rho-\bar \rho\|_{L^q(I)} \|v'\|_{L^{q'}(I)} \\
&\leq C \varepsilon \|\rho-\bar \rho\|_{L^q(I)} \|u'\|_{L^{p}(I)}^p
\end{align*}
and the result follows.
\end{proof}

Thanks to Proposition \ref{prop1d} we can proceed analogously as in the proof of Theorem \ref{teo.d} and Corollary \ref{cor.d} to conclude the following result:
\begin{cor} \label{cor.d.1}
Let $p> 1$ and $\{\rho_\varepsilon\}_{\varepsilon\in(0,1)} \in L^q(\Omega)$ with $1\leq q\leq \infty$. Let $\lambda_{k,\varepsilon}$ and $\lambda_{k,0}$ be the $k-$th eigenvalue of problems \eqref{eq.p.eps} and \eqref{eq.p.0} with $n=1$, $V=0$ and $\Omega=(0,1)\subset \mathbb R$,  respectively. Then
$$
\lim_{\varepsilon\to 0} \lambda_{k,\varepsilon} = \lambda_{k,0}.
$$
If additionally $\rho_\varepsilon:=\rho(\tfrac{x}{\varepsilon})$, where $\rho\in L^q(\Omega)$ is a positive  $1-$periodic function, and $\varepsilon \leq \min\{1,(\textbf{C} \lambda_{k,0})^{-1}\}$, being $\textbf{C}>0$  a computable constant independent of $k$ and $\varepsilon$, then
\begin{align*}
|\lambda_{k,0}-\lambda_{k,\varepsilon}| \leq C \|\rho\|_{L^q(\Omega)}\varepsilon  k^{2p}
\end{align*}
where $C = C (p)$ is a positive constant independent of $k$ and $\varepsilon$.
\end{cor}

Finally, by using the reduction given in Proposition \ref{cv} together with Corollary \ref{cor.d.1}, one gets the following:
\begin{thm} \label{teo.1d}
Let $\lambda_{k,\varepsilon}$ and $\lambda_{k,0}$ eigenvalues of problem \eqref{eq.p.a} and \eqref{eq.p.a.lim}, respectively. Then $\lim_{\varepsilon \to 0} \lambda_{k,\varepsilon}=\lambda_{k,0}$ and
\begin{align*}
|\lambda_{k,0}-\lambda_{k,\varepsilon}|\leq 
\begin{cases}
C L^{-p} \|a\|_{L^\infty(I)}^\frac{1}{p-1} \|\rho\|_{L^q(\Omega)}\varepsilon  k^{2p}  &\text{ if } \varepsilon^{-1}\in \mathbb N\\
C L^{-p} \left( \|a\|_{L^\infty(I)}^\frac{1}{p-1} \|\rho\|_{L^q(\Omega)}\varepsilon   k^{2p} +  k^p \varepsilon \right) &\text{ if } \varepsilon^{-1}\not \in \mathbb N.
\end{cases}
\end{align*}
holds for  $\varepsilon \leq \min\{ 1, (\textbf{C}\lambda_{k,0})^{-1}\}$, where $L=\overline{a^{\frac{-1}{p-1}}}$ and $C$ and  ${\bf C}$ are computable positive constants  independent of $k$ and $\varepsilon$.
\end{thm}

\begin{proof}
With the notation as in the proof of Proposition \ref{cv}, given $z\in(0,1):=I$
$$
g(z)= Q(Lz)= a(P^{-1}(Lz))^\frac{1}{p-1} \rho(P^{-1}(Lz)).
$$
Then, by Proposition \ref{cv} and Corollary \ref{cor.d.1} we get
$$
|\mu_{k,0}-\mu_{k,\delta}| \leq C \|a\|_{L^\infty(I)}^\frac{1}{p-1}\|\rho\|_{L^q(\Omega)}\delta^{1-\frac{1}{q}}  k^{2p}.
$$
holds for $\delta \leq \min\{1,(\textbf{C} \mu_{k,0})^{-1}\}$ for some computable constant $C$.

With the notation as in the proof of Proposition \ref{cv}, we recall that $\mu_{k,\delta}=L^p_\varepsilon \lambda_{k,\varepsilon}$, $\mu_{k,0}= L^p \lambda_{k,0}$, and $\delta=\varepsilon\frac{L}{L_\varepsilon}$, where  \eqref{condL} is fulfilled.
Then, when $\varepsilon^{-1}\in \mathbb N$ we have that $\delta=\varepsilon$ and $L_\varepsilon=L$, so
$$
|\lambda_{k,0}-\lambda_{k,\delta}| \leq {\bf C} \varepsilon  L^{-p} \|a\|_{L^\infty(I)}^\frac{1}{p-1} \|\rho\|_{L^q(\Omega)}    k^{2p}.
$$
When $\varepsilon^{-1}\not\in \mathbb N$, 
\begin{align*}
|\lambda_{k,\varepsilon}-\lambda_{k,0}|&= |L^{-p}(\mu_{k,\delta} - \mu_{k,0})-\mu_{k,\delta}(L_\varepsilon^{-p}-L^{-p})|\\
&\leq L^{-p} |\mu_{k,\delta}-\mu_{k,0}| + L^{-p} \lambda_{k,\varepsilon} |L^p-L_\varepsilon^p|.
\end{align*}
From \eqref{condL} it follows that $\left|\left(\frac{L_\varepsilon}{L}\right)^p -1\right| \leq p (1+\varepsilon)^{p-1}\varepsilon$, giving that
$$
|L^p-L_\varepsilon^p|\leq pL^p (1+\varepsilon)^{p+1}\varepsilon.
$$
Moreover, again from \eqref{condL} we have that $\varepsilon < \delta (1+\varepsilon)$. Then, assuming that $\varepsilon<1$,
$$
\varepsilon< 2\delta \leq 2\min\{1,(C \mu_{k,0})^{-1}\} = 2\min\{1,(C L^p \lambda_{k,0})^{-1}\}.
$$
This concludes the proof.
\end{proof}

\section*{Acknowledgements}
This paper is partially supported by grants UBACyT 20020130100283BA, CONICET PIP 11220150100032CO and ANPCyT PICT 2012-0153. Part of this article was written while the author was a professor at the Universidad de Buenos Aires, Argentina. The author want to thank to Prof. Juli\'an Fern\'andez Bonder for his suggestion to improve Proposition \ref{prop.dirichlet} in the case $p\leq n$.

\section*{Conflict of Interest Statement}
The authors declare that there is no conflict of interest regarding the publication of this paper.

\section*{Data Availability Statement}
No new data were created or analyzed in this study.

\end{document}